\documentclass[11pt]{article}
\usepackage{latexsym,bm}
\usepackage{mathrsfs}
\usepackage{amsmath,amsthm}
\usepackage{graphicx}
\usepackage{amssymb}
\usepackage{CJK}
\usepackage{color}
\usepackage{cite}
\usepackage{array}
\usepackage{stfloats}
\usepackage[colorlinks,
            linkcolor=blue,       
            anchorcolor=blue,  
            citecolor=blue,        
            ]{hyperref}
\usepackage{caption}
\usepackage{graphicx}
\usepackage{epstopdf}
\usepackage{upgreek}

\theoremstyle{plain}
\newtheorem{thm}{Theorem}[section]
\newtheorem{cor}[thm]{Corollary}
\newtheorem{lem}[thm]{Lemma}

\theoremstyle{definition}
\newtheorem{defn}[thm]{Definition}
\theoremstyle{plain}

\theoremstyle{problem}
\newtheorem{prob}{Problem}

\theoremstyle{plain}
\newtheorem{conj}{Conjecture}

\theoremstyle{plain}

\theoremstyle{plain}

\theoremstyle{plain}

\DeclareGraphicsExtensions{.eps,.eps.gz}
\usepackage[top=2.5cm,bottom=2.5cm,left=2.8cm,right=2.2cm]{geometry}
\normalsize \rm
\allowdisplaybreaks[4]
\makeatletter
\@addtoreset{equation}{section}
\makeatother

 \usepackage{indentfirst}
\begin{document}
\begin{CJK}{GBK}{song}
\newcommand{\song}{\CJKfamily{song}}    
\newcommand{\fs}{\CJKfamily{fs}}        
\newcommand{\kai}{\CJKfamily{kai}}      
\newcommand{\hei}{\CJKfamily{hei}}      
\newcommand{\li}{\CJKfamily{li}}        
\renewcommand\figurename{Fig.}

\begin{center}
{{\huge On the minimum constant resistance curvature conjecture of graphs}} \\[18pt]
{\Large Wensheng Sun$^{1}$, Yujun Yang$^{2}$, \footnotetext{*Corresponding author at E-mail address: shjxu@lzu.edu.cn } Shou-Jun Xu$^{1}$* }\\[6pt]
{ \footnotesize  $^{1}$ School of Mathematics and Statistics, Gansu Center for Applied Mathematics, Lanzhou University, Lanzhou, Gansu 730000 China\\
$^{2}$ School of Mathematics and Information Science,Yantai University, Yantai 264005 China}
\end{center}
\vspace{1mm}
\begin{abstract}

Let $G$ be a connected graph with $n$ vertices. The resistance distance $\Omega_{G}(i,j)$ between any two vertices $i$ and $j$ of $G$ is defined as the effective resistance between them in the electrical network constructed from $G$ by replacing each edge with a unit resistor. The resistance matrix of $G$, denoted by $R_G$, is an $n \times n$ matrix whose $(i,j)$-entry is equal to $\Omega_{G}(i,j)$. The resistance curvature $\kappa_i$ in the vertex $i$ is defined as the $i$-th component of the vector $(R_G)^{-1}\mathbf{1}$, where $\mathbf{1}$ denotes the all-one vector. If all the curvatures in the vertices of $G$ are equal, then we say that $G$ has constant resistance curvature.  Recently, Devriendt, Ottolini and Steinerberger \cite{kde} conjectured that the cycle $C_n$ is extremal in the sense that its curvature is minimum among graphs with constant resistance curvature. In this paper, we confirm the conjecture. As a byproduct, we also solve an open problem proposed by Xu, Liu, Yang and Das \cite{kxu} in 2016. Our proof mainly relies on the characterization of maximum value of the sum of resistance distances from a given vertex to all the other vertices in 2-connected graphs.

\vspace{1mm}
\end{abstract}

\section{Introduction}

Three decades ago, inspired by electrical network theory, Klein and Randi\'{c} \cite{djk1} proposed a novel distance function called resistance distance. Let $G= (V(G),E(G))$ be a connected graph with $n$ vertices. The \emph{resistance distance} between two vertices $i$ and $j$ of $G$, denoted by $\Omega_{G}(i,j)$, is defined to be the potential difference generated  between $i$ and $j$ induced by the unique $i \rightarrow j$ flow when the unit current flows in from node $i$ and flows out from node $j$. The \emph{resistance matrix}  $R_G$ of $G$ is an $n \times n$ matrix  such that its $(i,j)$-entry equal to $\Omega_{G}(i,j)$. For a vertex $u \in V(G)$, the \emph{resistive eccentricity index} of $u$, denoted by $\Omega_{G}(u)$, is defined as the sum of the resistance distances between $u$ and all the other vertices of $G$, that is:
$$\Omega_{G}(u) =\sum\limits_{\ v \in V(G)}\Omega_{G}(u,v).$$
A connected graph $G$ is called \emph{resistance-regular} \cite{jzh} if the  resistive eccentricity index of each vertex in $G$ is equal. The \emph{Kirchhoff index} of $G$ \cite{djk1}, denoted by $Kf(G)$, is defined as the sum of the resistance distances between all pairs of vertices, i.e.,
\begin{equation}\label{eq1.1}
Kf(G)=\sum\limits_{\{u,v\}\subseteq V(G)}\Omega_{G}(u,v) = \dfrac{1}{2}\sum\limits_{u\in V(G)}\Omega_{G}(u).
\end{equation}
As intrinsic metrics of graphs and classic components of circuit theory, resistance distance and Kirchhoff index have been widely studied, the readers are referred to recent papers \cite{jg1,jhu1,wsu,sxu1,sss,wko} and references therein for more details.

Curvature is a fundamental and important concept in differential geometry, geometric analysis and probability theory. Starting with the work of Bakry-\'{E}mery \cite{dba}, various discrete notions of curvature have been defined and exploited to understand the geometric properties of graphs, such as Bakry-\'{E}mery curvature \cite{yli,pho,hli}, Lin-Lu-Yau curvature \cite{yli1,qhu}, Ollivier-Ricci curvature \cite{yol,dpb}, Steinerberger curvature \cite{sst}, etc. This is a very active field of research, see \cite{frk,yol,kde,qhu} for more detail and references therein.  It is worth mentioning that in 2024, Devriendt, Ottolini and Steinerberger \cite{kde}  introduced the novel graph curvature via resistance matrix,  and characterized a large number of desirable and interesting properties in terms of resistance curvature, such as diameter, spectral gap and Kirchhoff index. In this paper, we devote ourselves to the resistance curvature on graphs. The definition of resistance curvature as follows.

\begin{defn}
Let $G$ be a connected graph with $n$ vertices.  The  resistance curvature $\kappa_i \in \mathbb{R}$ in the vertex $v_i \in V(G)$ is defined by requiring the vector $\kappa=(\kappa_1, \kappa_2, \ldots, \kappa_n)^{T} \in \mathbb{R}^{n} $ to solve a system of linear equations
$$R_G\kappa=\mathbf{1},$$
where $R_G$ is the resistance matrix and  $\mathbf{1} \in \mathbb{R}^{n}$ is the $n$-vector containing all 1's.
\end{defn}

It is well known that $R_G$ is non-singular \cite{rbb,jzh}, so the solution $\kappa$ is unique. In fact, the notion of resistance curvature has a number of desirable properties, for example, connected graphs with nonnegative resistance curvatures are 1-tough \cite{kde,kde1}. In particular, if the resistance curvature of every vertex in the graph $G$ is the same constant, then we say that the graph $G$ has constant resistance curvature, and denoted this constant as $\mathcal{K}_G$. It is not hard to see that  a graph $G$ has constant resistance curvature if and only if $G$ is resistance-regular. For a graph $G$ with constant resistance curvature and for every vertex $u \in V(G)$, the following formula holds
\begin{equation}\label{eq1.2}
\mathcal{K}_G=\frac{1}{\Omega_G(u)}=\frac{n}{2Kf(G)}.
\end{equation}
In their original work, Devriendt, Ottolini and Steinerberger \cite{kde} also determined the constant resistance curvature on some common vertex-transitive graphs, such as the complete graph $K_n$, the cycle $C_n$, the hypercube $Q_n$ and the $d$-dimensional discrete tori $C_{n,d}$. In addition, they showed that if $G$ is a graph with constant resistance curvature, then
$$\mathcal{K}_G \geq \frac{1}{n(n-1)}.$$
However, this lower bound seems too be somewhat rough and it is not sharp. Obviously, for $n=2$, then $\mathcal{K}_{K_2}=1$. For $n \geq 3$, they further proposed the following conjecture (see page 7 in \cite{kde}).
\begin{conj}\cite{kde}\label{conj1}
Let $G$ be a connected graph with constant resistance curvature on $n \geq 3$ vertices. Then
$$\mathcal{K}_G\geq \frac{6}{n^2-1},$$
with equality holding if and only if  $G=C_n.$
\end{conj}
A \emph{cut vertex} of a connected graph is a vertex whose deletion results in a disconnected graph. A graph is \emph{2-connected} if it is connected, has no cut vertices, and contains at least three vertices. Consider the family $\mathcal{C}_{n,k}$ of all connected graphs on $n$ vertices which have $k$ cut vertices. In 2016, Xu, Liu, Yang and Das \cite{kxu} characterized graphs in $\mathcal{C}_{n,k}$ with minimal Kirchhoff index for the case that $k\leq \frac{n}{2}$, and Nikseresht \cite{ani} characterized graphs in $\mathcal{C}_{n,k}$ with minimal Kirchhoff index for the case that $k\leq \frac{n}{2}$ and $k\geq n-3$. Recently, Huang, Huang, Liu and He \cite{jhu1} generalized their results and characterized graphs in $\mathcal{C}_{n,k}$ with minimal Kirchhoff index for the case that $0\leq k\leq n-2$. Although graphs in $\mathcal{C}_{n,k}$ with minimal Kirchhoff index have been completely characterized, there are few results on extremal graphs in $\mathcal{C}_{n,k}$ with maximal Kirchhoff index. Even for the case that $k=0$ (i.e., 2-connected graphs), it is still open, even though Xu, Liu, Yang and Das \cite{kxu} guess that the cycle $C_n$ is a natural candidate.
\begin{prob}\cite{kxu} \label{conj2}
Characterize extremal graphs with maximal Kirchhoff index among all 2-connected graphs.
\end{prob}
Although Conjecture \ref{conj1} and Problem \ref{conj2} appear unrelated at first glance, we could establish a close relation between them. In this paper, we are able to show that graphs with constant resistance curvature and $n \geq 3$ vertices must be 2-connected. We then show that the cycle $C_n$ is the unique graph among all 2-connected graphs in which every vertex has the maximum resistive eccentricity index. As a direct consequence, we confirm Conjecture \ref{conj1} and solve Problem \ref{conj2}.

\section{Proof of the main result}
In this section,  we first introduce some basic notations that will be used later.

Let $G= (V(G),E(G))$ be a connected graph, then $|V(G)|$ and $|E(G)|$ are called the order and size of $G$, respectively. We use $d_G(u,v)$ to denote the (shortest path) distance between two vertices $u$ and $v$ of $G$.
For edge $e \in E(G)$ and $u \in V(G)$, we use $G-e$ and $G-u$ to denote the graph obtained from $G$ by deleting edge $e$ and vertex $u$ and all edges incident to $u$, respectively. A \emph{subgraph} of $G$ is a graph $H=(V(H),E(H))$ where $V(H) \subseteq V(G)$ and $E(H) \subseteq E(G)$. For a subgraph $H$, if $V(H)=V(G)$, then $H$ is a \emph{spanning subgraph} of $G$. For vertex $U \subseteq V(G)$, then the subgraph consisting of $U$ and all the edges of $G$ that join two vertices of $U$ is called an \emph{induced subgraph} induced by $U$. Let $C$ be a cycle in $G$. A \emph{chord} of a cycle $C$ is an edge that joins two non-adjacent vertices on this cycle $C$.

Then we  introduce some common tools and principles in electrical network theory. Since a resistor on network can always be viewed as a weighted edge of its corresponding graph, we do not distinguish between electrical networks and the corresponding graphs.

A \emph{block} of $G$ is a maximal connected subgraph of $G$ that has no cut vertex. Obviously, if $G$ is a 2-connected graph, then itself is a block. Suppose that $x$ is a cut vertex and $H_1$ and $H_2$ are subgraphs of $G$. Then we say that $(H_1, H_2)$ is an $x$-\emph{separation} of $G$ if $G = H_1 \cup H_2$ and $V(H_1) \cap V(H_2)=\{x\}$. The following principle can simplify the calculation of resistance distances for graphs with cut vertex.

\textbf{Principle of elimination}\cite{ani1}. Suppose that $G$ is a connected graph and $B$ is a block that contains exactly one cut vertex $x$ in $G$, then the subgraph $H$ obtained from $G$ by deleting all the vertices of $B$ except $x$ satisfies $\Omega_{H}(u, v) = \Omega_{G}(u, v)$ for $u,v\in V(H)$.
\begin{lem}\cite{djk1} \label{lem2.11}
Let $x$ be a cut vertex of graph $G$ and $(H_1, H_2)$ be an $x$-\emph{separation} of $G$. Then for any vertex $u \in V(H_1)$ and $v \in V(H_2)$
\begin{equation*} \Omega_G(u,v)= \Omega_{H_1}(u,x)+ \Omega_{H_2}(x,v).
\end{equation*}
\end{lem}

\textbf{Rayleigh's monotonicity law}\cite{pgd}. In an electrical network, if the edge-resistance increases, then the effective resistance between any pair of nodes will not decrease.

For a connected graph $G$, if we delete an edge (resp., vertex) from $G$, then it means that the resistance on the edge (resp., all edges incident to the vertex) will increase from $1$ to $+ \infty$. Thus, by Rayleigh's monotonicity law, the resistance distance between any two vertices of $G$ will not decrease. Note that a subgraph of a graph $G$ could be obtained from $G$ by iteratively deleting vertices and edges. Therefore, we have the following results.
\begin{lem}\label{lem2.1}
Let $G$ be a graph and $H$ be a subgraph of $G$. Then for $u,v \in V(H)$, we have
\begin{center}$\Omega_{G}(u, v) \leq \Omega_{H}(u, v)$. \end{center}
\end{lem}

\begin{lem}\label{lem2.2}
Let $G$ be a graph and $H$ be a spanning subgraph of graph $G$. Then for $u \in V(G)$,
\begin{center}$\Omega_{G}(u) \leq \Omega_{H}(u)$. \end{center}
\end{lem}

In \cite{yya}, Yang and Klein obtained a recursion formula for resistance distances on weighted graphs. If only one edge in the graph is deleted, the calculation formula for the resistance distance is as follows.
\begin{thm}\cite{yya}\label{thm2.4}
Let $G$ be a connected graph and edge $e=ij\in E(G)$. Then for vertex pair $p,q\in V(G)$ and $G'=G-e$, we have
 $$\Omega_{G}(p,q)=\Omega_{G'}(p,q)-\frac{[\Omega_{G'}(p,i)+\Omega_{G'}(q,j)-\Omega_{G'}(p,j)-\Omega_{G'}(q,i)]^{2}}{4[1+ \Omega_{G'}(i,j)]}.$$
 \end{thm}

Now, we use combinatorial and electrical network techniques to prove that all graphs with constant resistance curvature and $n \geq 3$ vertices must be 2-connected graphs.
\begin{thm}\label{tm3.1}
Let $G$ be a connected graph with constant resistance curvature on $n \geq 3$ vertices. Then $G$ is a 2-connected graph.
\end{thm}
\begin{proof}
\begin{figure*}[ht]
  \setlength{\abovecaptionskip}{0cm} %
  \centering
 $$\includegraphics[width=2in]{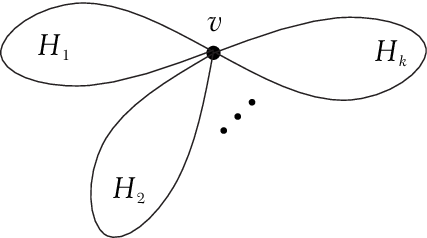}$$\\
 \caption{The graph $G$ in the proof of Theorem \ref{tm3.1}.}    \label{Fig.1}
\end{figure*}
Suppose to the contrary that there exists cut vertex $v \in V(G)$, and $G-v$ has  $G_{1}, G_{2},...,G_{k}$ connected components, where $k\geq 2$.  Let $|V(G_i)| = n_i$, and the subgraph induced by vertex set $V(G_i) \cup \{v\}$ of $G$ is $H_i$,  where $1 \leq i \leq k$. See Fig. \ref{Fig.1}. Without loss of generality, suppose that $n_1 \leq \sum\limits_{i=2}^kn_i.$  Since $V(G)=V(H_1) \cup V(G_2)\cup \cdots \cup V(G_k)$, for each vertex $u \in V(G_1)$, by the principle of elimination and Lemma \ref{lem2.11}, we have
\begin{align}\label{eq3.1}
\Omega_{G}(u)&=\sum_{w \in V(G)}\Omega_{G}(u,w) \notag \\
&=\sum_{w \in V(H_1)}\Omega_{G}(u,w)+\sum_{i=2}^{k}\sum_{w \in V(G_i)}\Omega_{G}(u,w)\notag \\
&=\sum_{w \in V(H_1)}\Omega_{H_1}(u,w)+\sum_{i=2}^{k}\sum_{w \in V(G_i)}\Big(\Omega_{H_1}(u,v)+\Omega_{H_i}(v,w)\Big)\notag \\
&=\Omega_{H_1}(u)+\sum_{i=2}^kn_i \Big(\Omega_{H_1}(u,v)\Big)+\sum_{i=2}^{k}\Omega_{H_i}(v).
\end{align}
It is obvious that
\begin{equation}\label{eq3.2}
\Omega_{G}(v)=\sum_{i=1}^{k}\Omega_{H_i}(v).
\end{equation}
Since $G$ has the constant resistance curvature, we know $\Omega_{G}(u)=\Omega_{G}(v)$. Thus, comparing Eqs. (\ref{eq3.1}) and (\ref{eq3.2}), we get
\begin{equation}\label{eq3.3}
\Omega_{H_1}(u)+\sum_{i=2}^kn_i \Big(\Omega_{H_1}(u,v)\Big)=\Omega_{H_1}(v).
\end{equation}
By summing up both sides of Eq. (\ref{eq3.3}) for all the vertices of $V(G_1)$, we get
\begin{equation}\label{eq3.4}
\sum_{u \in V(G_1)}\Omega_{H_1}(u)+\sum_{u \in V(G_1)}\sum_{i=2}^kn_i \Big(\Omega_{H_1}(u,v)\Big)=\sum_{u \in V(G_1)}\Omega_{H_1}(v).
\end{equation}
Note that
\begin{equation}\label{eq3.5}
\sum_{u \in V(G_1)}\sum_{i=2}^kn_i \Big(\Omega_{H_1}(u,v)\Big)=\sum_{i=2}^kn_i \Big(\Omega_{H_1}(v) \Big).
\end{equation}
Substituting Eq. (\ref{eq3.5}) into Eq. (\ref{eq3.4}) yields
\begin{align}
\sum_{u \in V(G_1)}\Omega_{H_1}(u)&=n_1\Omega_{H_1}(v)-\sum_{i=2}^kn_i \Big(\Omega_{H_1}(v) \Big) \notag \\
&=\left(n_1-\sum_{i=2}^kn_i \right)\Big(\Omega_{H_1}(v) \Big) \notag \\
&\leq 0, \notag
\end{align}
which contradicts to the fact that $\sum_{u \in V(G_1)}\Omega_{H_1}(u)>0$ since $H_1$ has at least two vertices. Thus we conclude that $G$ is a 2-connected.
\end{proof}

\begin{thm}\label{tm3.2}
Let $G$ be a 2-connected graph with $n \geq 3$ vertices. Then for every vertex $u \in V(G)$ we have
$$\Omega_{G}(u) \leq \frac{n^2-1}{6}. $$
Moreover, the bound is achieved if and only if  $G=C_n.$
\end{thm}
\begin{proof}
If $n=3$, then the only 2-connected graph on 3 vertices is the complete graph $K_3$ and the result follows obviously. So we suppose that $n \geq 4$. Since $G$ is 2-connected, $G$ contains a cycle. Let $u$ be an arbitrary vertex of $G$. Let $C_k$ be the longest cycle in $G$ containing $u$. We distinguish the following two cases according to the length of $C_k$.

\textbf{Case 1.} $k=n$. In this case, $G$ contains a Hamiltonian cycle $C_n$. If $G=C_n$, then the desired result follows directly. Otherwise, $G\neq C_n$ and $G$ could be obtained by adding chords to $C_n$.  Let $G'$ be a spanning subgraph of $G$ which is obtained from $C_n$ by adding exactly one chord.  For any vertex $u \in V(G')=V(C_n)$, we know that $\Omega_{C_n}(u)= \frac{n^2-1}{6}$.  Hence, in the following, we show that $\Omega_{G'}(u)< \Omega_{C_n}(u)$, so that by Lemma \ref{lem2.2} we have $\Omega_{G}(u)\leq \Omega_{G'}(u)<\frac{n^2-1}{6}$ as desired.

Let $ij\in E(G')$ be the unique chord of $C_n$. Without loss of generality, suppose that $i \neq u$. Then according to Theorem \ref{thm2.4}, we have
\begin{equation}\label{eq2.6}
\Omega_{G'}(u,i)=\Omega_{C_n}(u,i)-\frac{[\Omega_{C_n}(u,i)+\Omega_{C_n}(i,j)-\Omega_{C_n}(u,j)]^{2}}{4[1+ \Omega_{C_n}(i,j)]}.
\end{equation}
Since $i$ is not a cut vertex of $G'$ that separates $u$ and $j$, it follows by the triangular inequality of the resistance distance that
\begin{equation}\label{eq2.7}
\Omega_{C_n}(u,i)+\Omega_{C_n}(i,j)-\Omega_{C_n}(u,j)>0.
\end{equation}
Combining Eq. (\ref{eq2.6}) and Ineq. (\ref{eq2.7}), we know that
\begin{equation}
\Omega_{G'}(u,i)<\Omega_{C_n}(u,i).
\end{equation}
On the other hand, for any vertex $v \in V(G')\backslash \{u\}$, by Lemma \ref{lem2.1}, we have
\begin{equation*}
 \Omega_{G'}(u, v) \leq  \Omega_{C_n}(u, v).
\end{equation*}
Thus we get
\begin{align}
\Omega_{G'}(u) - \Omega_{C_n}(u)&= \sum_{v \in V(G)}\Omega_{G}(u)- \sum_{v \in V(C_n)}\Omega_{C_n}(u)\notag \\
&=\Big(\Omega_{G'}(u,i)-\Omega_{C_n}(u,i)\Big) + \sum_{v \in V(G)\backslash \{u\}} \Big(\Omega_{G'}(u, v)- \Omega_{C_n}(u, v)\Big) \notag \\
&< 0.\notag
\end{align}
Therefore, by Lemma \ref{lem2.2}, we can conclude
\begin{equation*}
\Omega_{G}(u) \leq \Omega_{G'}(u) < \Omega_{C_n}(u)=\frac{n^2-1}{6}.
\end{equation*}

\begin{figure*}[ht]
  \setlength{\abovecaptionskip}{0cm} %
  \centering
 $$\includegraphics[width=5.0in]{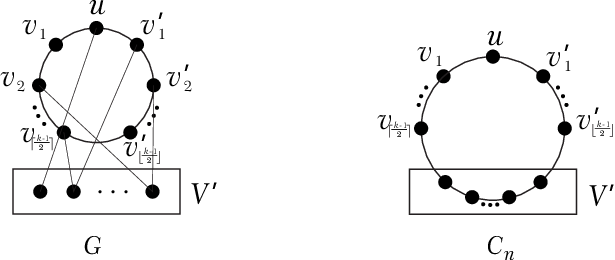}$$\\
 \caption{ Illustration of vertex labeling of graphs $G$ and $C_n$ in Case 2.}    \label{Fig.2}
\end{figure*}

\textbf{Case 2.} $k<n$.  In the following, we show $\Omega_{G}(u)<\frac{n^2-1}{6}$.

Since $C_k$ be the longest cycle in $G$ containing $u$. We first label the vertices of $G$ in the following way:
vertices lying on $C_k$ are labeled by $C_k:=uv_1v_2\cdots v_{\lceil\frac{k-1}{2}\rceil}v'_{\lfloor\frac{k-1}{2}\rfloor}\cdots v'_2v'_1u$ as shown in Fig. \ref{Fig.2}, and the remaining vertices are labeled in an arbitrary manner. In order to compare resistance distances between vertices in $G$ and $n$-cycle $C_n$, we use the same labeling of $V(G)$ to label vertices in $C_n$: first choose an arbitrary vertex of $C_n$ and label it as $u$,  then label the vertices at distance $i$ to $u$  in $C_n$ by $v_i$ for $1 \leq i \leq \lceil\frac{k-1}{2}\rceil$,  and label the vertices at distance $i$ to $u$ by $v'_i$ for $1 \leq i \leq \lfloor\frac{k-1}{2}\rfloor$,  as shown in Fig. \ref{Fig.2}.  Next we show that for any vertex $v\in V(G)=V(C_n)$ other than $u$, $\Omega_{G}(u,v)<\Omega_{C_n}(u,v)$.

For any vertex $v_i$ ($1 \leq i \leq \lceil\frac{k-1}{2}\rceil$), by Lemma \ref{lem2.1}, we have
\begin{equation}\label{eq3.2.3}
\Omega_{G}(u,v_i) \leq \Omega_{C_k}(u,v_i)=\frac{(k-i)i}{k} < \frac{(n-i)i}{n}= \Omega_{C_n}(u,v_i).
\end{equation}
In the same way, for vertex $v'_i$ ($1 \leq i \leq \lfloor\frac{k-1}{2}\rfloor$), we also have
\begin{equation}\label{eq3.2.4}
\Omega_{G}(u,v'_i) < \Omega_{C_n}(u,v'_i).
\end{equation}
Then Ineqs. (\ref{eq3.2.3}) and (\ref{eq3.2.4}) imply that
\begin{equation}\label{eq1}
\sum_{v \in V(C_k)} \Omega_{G}(u, v) < \sum_{v \in V(C_k)} \Omega_{C_n}(u, v).
\end{equation}
Now let $V'$ be the set of vertices of $G$ not lying on $C_k$, that is $V'=V(G)\setminus V(C_k)$.  Choose $w\in V'$ such that
$$\Omega_{G}(u,w) = \mbox{max} \Big\{ \Omega_{G}(u,v)\ | \ v \in V' \Big\}.   $$
Since $G$ is 2-connected, $u$ and $w$ must lie on a common cycle $C_m$ in $G$ and by the assumption of $C_k$ we know that $m\leq k$. If $m$ is even, it is easily seen that
$$\Omega_{G}(u, w) \leq \Omega_{C_m}(u, w) \leq \frac{m}{4}.$$
Otherwise, if $m$ is odd, then
$$\Omega_{G}(u, w) \leq   \Omega_{C_m}(u, w) \leq \frac{m}{4}-\frac{1}{4m}.$$
To sum up, no matter the parity of $m$, we always have
$$\Omega_{G}(u, w) \leq   \frac{m}{4} \leq \frac{k}{4}.$$
Thus for any vertex $v\in V'$, we have
\begin{equation}\label{eq2.12}
\Omega_{G}(u, w)\leq \frac{k}{4}.
\end{equation}

On the other hand, for any $v \in V'$ in $C_n$, let $P$ the path between $v$ and $v_{\lceil\frac{k-1}{2}\rceil}$ not passing $u$ in $C_n$, and $Q$ be the path between vertices $v$ and $v'_{\lfloor\frac{k-1}{2}\rfloor}$ not passing $u$ in $C_n$. Suppose that the lengths of $P$ and $Q$ are $x$ and $y$, respectively. Obviously, $x \geq 1$ and $y \geq 1$. By series and parallel connection rules, for the case where $k$ is odd, we have
\begin{align}\label{eq3.2.5}
\Omega_{C_n}(u,v)&=\frac{\left(\frac{k-1}{2}+x\right)\left(\frac{k-1}{2}+y\right)}{n}\notag \\
&= \frac{\frac{(k-1)^2}{4}+\frac{(k-1)(x+y)}{2}+xy}{k-1+x+y}            \notag \\
&= \frac{\frac{k}{4}(k-1+x+y)+[\frac{k}{4}(x+y-1)+xy-\frac{x+y}{2}+\frac{k}{4}]}{k-1+x+y}  \notag \\
&> \frac{k}{4} \quad\quad\quad \mbox{(as $\frac{k}{4}(x+y-1)+xy-\frac{x+y}{2}+\frac{1}{4}>0$)}.
\end{align}
Similarly, for the case where $k$ is even, we get
\begin{align}\label{eq3.2.14}
\Omega_{C_n}(u,v)&=\frac{\left(\frac{k}{2}+x\right)\left(\frac{k}{2}-1+y\right)}{n}\notag \\
&= \frac{\frac{k^2}{4}+\frac{k(x+y)}{2}+xy-x-\frac{k}{2}}{k-1+x+y}       \notag \\
&= \frac{\frac{k}{4}(k-1+x+y)+[\frac{k}{4}(x+y)+xy-x-\frac{k}{4}]}{k-1+x+y}  \notag \\
&> \frac{k}{4} \quad\quad\quad \mbox{(as $\frac{k}{4}(x+y)+xy-x-\frac{k}{4}>0$)}.
\end{align}
By Ineqs. (\ref{eq2.12})-(\ref{eq3.2.14}), we know that no matter the parity of $k$, for every $ v \in V'$, $\Omega_{G}(u, v)<\Omega_{C_n}(u,v)$ always holds. Thus it follows that
\begin{equation}\label{eq3.2.6}
\sum_{v \in V'} \Omega_{G}(u, v) < \sum_{v \in V'} \Omega_{C_n}(u, v).
\end{equation}
Combining Ineqs. (\ref{eq1}) and (\ref{eq3.2.6}), we get
\begin{equation*}
\Omega_{G}(u) < \Omega_{C_n}(u)= \frac{n^2-1}{6}.
\end{equation*}
The proof is complete.
\end{proof}

Combining Eq. (\ref{eq1.2}) and Theorem \ref{tm3.2}, we could confirm Conjecture \ref{conj1}.
\begin{thm} \label{thm4.2}
Let $G$ be a connected graph with constant resistance curvature on $n \geq 3$ vertices. Then we have
$$\mathcal{K}_G\geq \frac{6}{n^2-1},$$
with equality holding if and only if  $G=C_n.$
\end{thm}
In \cite{kde}, Devriendt, Ottolini and Steinerberger proved that the complete graph $K_n$ (i.e., $\mathcal{K}_{K_n}=n/2n-2$) has the maximum curvature among graphs with constant resistance curvature. Together with Theorem \ref{thm4.2}, we have
\begin{cor}
Let $G$ be a connected graph with constant resistance curvature on $n \geq 3$ vertices. Then we have
$$\frac{6}{n^2-1} \leq \mathcal{K}_G  \leq \frac{n}{2n-2},$$
with the left equality holding if and only if  $G=C_n$, and the right equality holding if and only if  $G=K_n$.
\end{cor}
According to Theorem \ref{tm3.2}, we also give the solution to Problem \ref{conj2}.
\begin{thm}
Let $G$ be a 2-connected graph with $n$ vertices. Then we have
$$Kf(G) \leq \frac{n^3-n}{12}, $$
with equality holding if and only if  $G=C_n.$
\end{thm}
It is well-known that the complete graph $K_n$ (i.e., $Kf(K_n)=n-1$) is the unique graph with the minimum Kirchhoff index  among all connected graphs with $n$ vertices. As a corollary, we characterize the extremal resistance-regular graphs with respect to the Kirchhoff index.
\begin{cor}
Let $G$ be a resistance-regular graph with $n$ vertices. Then we have
$$  n-1  \leq Kf(G) \leq \frac{n^3-n}{12}, $$
with the left equality holding if and only if  $G=K_n$, and the right equality holding if and only if  $G=C_n$.
\end{cor}

\section{Declaration of competing interest}
The authors declare that they have no known competing financial interests or personal relationships that could have appeared to influence the work reported in this paper.
\section{Data availability}
No data was used for the research described in the article.
\section{Acknowledgments}
The support of the National Natural Science Foundation of China (through grant no. 12071194) and Taishan Scholars Special Project of Shandong Province is greatly acknowledged.

\end{CJK}

\begin{thebibliography}{9}
{\small

\bibitem{dba}D. Bakry, M. \'{E}mery, Diffusions hypercontractives, in S\'{e}minaire de Probabilit\'{e}s, XIX, 1983/84, 177--206, Lecture Notes in Math. 1123, Springer, Berlin, (1985).

\bibitem{rbb}R.B. Bapat, Resistance matrix of a weighted graph, MATCH Commun. Math. Comput. Chem. 50(02) (2004) 73--82.

\bibitem{dpb}D.P. Bourne, D. Cushing, S. Liu, et al, Ollivier-Ricci idleness functions of graphs, SIAM J. Discrete Math. 32(2) (2018) 1408--1424.

\bibitem{frk}F.R.K. Chung, S-T. Yau, Logarithmic harnack inequalities, Math. Res. Lett. 3(6) (1996) 793--812.

\bibitem{kde1}K. Devriendt, R. Lambiotte, Discrete curvature on graphs from the effective resistance, J. Phys.: Complexity 3 (2022) 025008.

\bibitem{kde}K. Devriendt, A. Ottolini, S. Steinerberger, Graph curvature via resistance distance, Discrete Appl. Math. 348 (2024) 68--78.

\bibitem{pgd}P.G. Doyle, J.L. Snell, Random Walks and Electric Networks, The Mathematical Association of America, Washington, DC, 1984.

\bibitem{jg1} J. Ge, Y. Liao, B. Zhang, Resistance distances and the Moon-type formula of a vertex-weighted complete split graph, Discrete Appl. Math. 359 (2024) 10--15.

\bibitem{pho}P. Horn, A. Purcilly, A. Stevens, Graph curvature and local discrepancy, J. Graph Theory 108(2) (2025) 337--360.

\bibitem{jhu1}J. Huang, G. Huang, J. Li, W. He, On the minimum Kirchhoff index of graphs with a given number of cut vertices, Discrete Appl. Math. 365 (2025) 27--38.

\bibitem{qhu}Q. Huang, W. He, C. Zhang, Graphs with positive Ricci curvature, Graphs Combin. 41(1) (2025) 14.

\bibitem{djk1} D.J. Klein, M. Randi\'{c}, Resistance distance, J. Math. Chem. 12 (1993) 81--95.

\bibitem{wko}W. Kook, K.J. Lee, Simplicial Kirchhoff index of random complexes, Adv. Appl. Math. 159 (2024) 102733.

\bibitem{hli}H. Lin, Z. You, Graphs with nonnegative Bakry-\'{E}mery curvature without quadrilateral, Proc. Amer. Math. Soc. 153(05) (2025) 1839--1848.

\bibitem{yli1}Y. Lin, L. Lu, S-T. Yau, Ricci curvature of graphs, Tohoku Math. J. 63(4) (2011) 605--627.

\bibitem{yli}Y. Lin, S-T. Yau, Ricci curvature and eigenvalue estimate on locally finite graphs, Math. Res. Lett. 17(2-3) (2010) 343--356.

\bibitem{ani}A. Nikseresht, On the minimum Kirchhoff index of graphs with a fixed number of cut vertices, Discrete Appl. Math. 207 (2016) 99--105.

\bibitem{ani1}A. Nikseresht, Z. Sepasdar, On the Kirchhoff and the Wiener indices of graphs and block decomposition, Electron. J. Combin. 21(1) (2014) P1.25.

\bibitem{yol}Y. Ollivier, Ricci curvature of Markov chains on metric spaces, J. Funct. Anal. 256(3) (2009) 810--864.

\bibitem{sss}S.S. Saha, S.K. Panda, On the Kirchhoff index of hypergraphs, Czech. Math. J. in press, https://doi.org/10.21136/CMJ.2025.0308-24

\bibitem{sst}S. Steinerberger, Curvature on graphs via equilibrium measures, J. Graph Theory 103(3) (2023) 415--436.

\bibitem{wsu}W. Sun, Y. Yang, S. Xu, Resistance distance and Kirchhoff index of unbalanced blowups of graphs, Discrete Math. 348(3) (2025) 114327.

\bibitem{kxu}K. Xu, H. Liu, Y. Yang, K.C. Das, The minimal Kirchhoff index of graphs with a given number of cut vertices, Filomat 30(13) (2016) 3451--3463.

\bibitem{sxu}S. Xu, K. Xu, Resistance distances in generalized join graphs, Discrete Appl. Math. 362 (2025) 18--33.

\bibitem{sxu1}S. Xu, H. Zhou, X. Pan, A method for constructing graphs with the same resistance spectrum, Discrete Math. 348(2) (2025) 114284.

\bibitem{yya}Y. Yang, D.J. Klein, A recursion formula for resistance distances and its applications. Discrete Appl. Math. 161(16-17) (2013) 2702--2715.

\bibitem{jzh}J. Zhou, Z. Wang, C. Bu, On the resistance matrix of a graph, Electron. J. Combin. (2016) P1.41.





}
\end{thebibliography}
\end{document}